\documentclass[12pt]{article}%
\usepackage{amsmath}
\usepackage{amsfonts}
\usepackage{amssymb}
\usepackage{graphicx}%
\providecommand{\U}[1]{\protect\rule{.1in}{.1in}}
\newtheorem{theorem}{Theorem}

\newtheorem{corollary}[theorem]{Corollary}

\newtheorem{definition}[theorem]{Definition}
\newtheorem{example}[theorem]{Example}

\newtheorem{lemma}[theorem]{Lemma}

\newtheorem{proposition}[theorem]{Proposition}

\newenvironment{proof}[1][Proof]{\noindent\textbf{#1.} }{\ \rule{0.5em}{0.5em}}
\begin{document}

\title{Preserving operators on semiprime $f$-algebras}
\author{\textsc{Jaber Jamel, Khalfaoui ADNEN}\\Research Laboratory of Algebra, Topology, Arithmetic and Order \ \\Departement of Mathematics Tunis-El Manar University.\\\texttt{{\small jamel.jaber@free.fr}}}
\maketitle

\begin{abstract}
It is an open problem whether a separating operator acting between semiprime
$f$-algebras is a weighted composition operator ( \cite[problem 2.6]{AAB}). We
prove that the answer is positive if \ and only if the separating operator is
almost contractive. As a consequence, we can generalized and strengthen some
well-known results on\ separating operators.

\end{abstract}



\section{Introduction}

Let $A$ and $B$ be two arbitrary algebras. A map $T:A\longrightarrow B$ is
said to be separating if $T(f)T(g)=0$ whenever $fg=0$ for all $f,g\in A$. Any
Composition operator (algebra homomorphism) is separating. Weighted
composition operators are important typical examples of separating map. The
question of whether or not a separating operator is weighted have been studied
by many authors in different cases. Arendt \cite{Arendt} has proved that every
order bounded separating linear map $T:C(X)\longrightarrow C(Y)$ is a multiple
of an algebra homomorphism, that is, there exist a function
$h:Y\longrightarrow X$ such that
\[
T(f)(y)=T(e).f\circ h(y)\text{ for all }y\in Y\text{ and }f\in C(X)\text{,}%
\]
where $e$ denotes the function $e(x)=1$ for all $x\in X$. Recently, Abid, Ben
Amor and Boulabiar \cite[Theorem 4.2 and corollary 4.3]{AAB} have been
generalized this result to a large class of regular separating operators
acting between unital $f$-algebras. More precisely, they proved that if $A$ is
an $f$-algebra with identity $e$ and $B$ is a semiprime $f$-algebra then a
regular operator $T:A\longrightarrow B$ is separating if and only if $T=Te.C$
with $C$ is a composition operator from $A$ into the the maximal ring of
quotients of $B$ (see \cite[corollary 4.3]{AAB}). Furthermore, they stated the
following open problem (see for instance \cite[problem 2.6]{AAB}) : Does the
last result hold if we suppose that the domain $f$-algebra is semiprime with
no identity? It should be pointed out by the way that this is true in the
$C_{0}(X)-C_{0}(Y)$ case (Theorem 1 in \cite{Jeang-wong}). Here, $C_{0}(X)$
(resp, $C_{0}(Y)$) is the semiprime $f$-algebra of continuous scalar-valued
functions on $X$ (resp, $Y$) vanishing at infinity. Unfortunately, the result
fails in general (see Example 6.6 in \cite{BBT}). The main purpose of this
paper is to give a necessary and sufficient condition for which a separating
regular linear map between semiprime $f$-algebras is a weighted composition
homomorphism. A synopsis of the content of this paper seems to be in order.

Let $A$ and $B$ be two semiprime $f$-algebras. By almost contractive operator
$T:A\longrightarrow B$ we mean a regular operators such that $T(\left[
0,I_{A}\right]  \cap A)$ is bounded, where $I_{A}$ denote the unit element of
the $f$-algebra $\mathrm{Orth}(A)$. Such operators can be seen as a
generalization of contractive operators acting between unital $f$-algebras. it
is not difficult to see that any weighted composition operator is automaticly
almost contractive. However, we give an example of a separating regular
operator $T$ from $A$ into $B$ which is not almost contractive. In spite of
that, we shall prove that any separating regular operator $T$ from $C_{0}(X)$
into $C_{0}(Y)$ is almost contractive. Using these remarks we shall prove that
a regular linear operator $T$ is a weighted composition operator if and only
if $T$ is separating and almost contractive operator. Moreover, we show that
this result generalized the afformentioned results. 

\section{ Preliminaries}

We take it for granted that the reader is familiar with the notions of vector
lattices (also called Riesz spaces) and regular operators such as lattice
homomorphisms and orthomorphisms. For terminology, notations and concepts not
explained in this paper we refer to the standard monographs \cite{Aliprantis}
by Aliprantis-Burkinshaw.

\textit{Beginning with the next paragraph, we shall impose as blanket
assumptions that all vector lattices under consideration are real and
Archimedean. Moreover, all given operators are supposed to be linear.}

An $f$-algebra $A$ is said to be semiprime $f$-algebra if $0$ is the only
nilpotent in $A$. Any unital $f$-algebra is semiprime. Moreover, the algebra
$C_{0}(X)$ of all real valued continuous functions defined on a locally
compact space $X$ vanishing at infinity is a typical example of semiprime $f$-algebra.

It is well-known that every semiprime $f$-algebra $A$ can be embedded as a
Riesz subspace and a ring ideal in the unital $f$-algebra $\mathrm{Orth}(A)$,
the set of all orthomorphisms on $A$, by identifying $f\in A$ with the
orthomorphism $\pi_{f}$ on $A$ defined by $\pi_{f}(g)=fg$ for all $g\in A$.
Notice here that the identity operator $I_{A}$ on $A$ is the unit element of
$\mathrm{Orth}(A)$. More information on the embedding of semiprime $f$
-algebras in algebras of orthomorphisms can be found in \cite[Section
12.3]{BK}.

We will denote by $U_{A}$ the set defined by
\[
U_{A}:=\left[  0,I_{A}\right]  \cap A=\left\{  f\in A:0\leq f\leq
I_{A}\right\}  =\left\{  f\in A:f^{2}\leq f\right\}  \text{.}%
\]
It is well known that the set $U_{A}$ is an approximate unit on $A$, that is
\[
\sup\left\{  fg:f\in U_{A}\right\}  =g\text{ for all }0\leq g\in A\text{,}%
\]
(see for instance \cite[Theorem 2.3]{HD1984}). As we shall see later, the set
$U_{A}$ will play an important role in this work.

We adopt in this paper the notations of \cite{AAB}. Recall that an operator
$T:A\longrightarrow B$ is said to be regular if $T$ is the difference of two
positive operators. The set $\mathcal{L}^{r}(A,B)$ of all regular operators is
an ordered vector space with respect to the pointwise addition, scalar
multiplication, and ordering.

The linear operator $T\in\mathcal{L}^{r}(A,B)$ is said to be a composition
operator if $T(fg)=T(f)T(g)$ for all $f,g\in A$. The operator $T$ is called a
separating (or disjointness preserving) operator $TfTg=0$ whenever $fg=0$ in
$A$. Since both $A$ and $B$ are semiprime, we derive that $T$ is separating if
and only if $\left\vert Tf\right\vert \wedge\left\vert Tg\right\vert =0$ in
$B$ whenever $\left\vert f\right\vert \wedge\left\vert g\right\vert =0$ in
$A$. It follows that a positive operator is separating if and only if $T$ is a
lattice homomorphism. If $T\in\mathcal{L}^{r}(A,B)$ a separating operator then
the modulus $\left\vert T\right\vert $ exists, satisfies $\left\vert
T\right\vert \left\vert f\right\vert =\left\vert Tf\right\vert $ for all $f\in
A$ and $\left\vert T\right\vert $ is a lattice homomorphism.

An operator $T\in\mathcal{L}^{r}(A,B)$ is said to be contractive if
$\left\vert Ta\right\vert \in U_{B}$ for all $a\in U_{A}$. An equivalent and
more intrinsic statement is that $\left\vert Ta\right\vert \leq I_{B}$ for all
$a\in U_{A}$. Notice, here that $I_{B}$ is the unit element of $\mathrm{Orth}%
(B)$. We shall extend this definition as follows.

\begin{definition}
Let $T\in\mathcal{L}^{r}(A,B)$. The operator $T$ is said to be almost
contractive if there exist $\omega\in\mathrm{Orth}(B)$ such that $\left\vert
Tf\right\vert \leq\omega$ for all $f\in U_{A}$. In this case we define the
vector $\omega_{T}$ by
\[
\omega_{T}:=\sup\left\vert T(U_{A})\right\vert =\sup\left\{  \left\vert
Tf\right\vert :f\in U_{A}\right\}  \text{.}%
\]
The supremum is taken, if necessarily, in the Dedekind completion
$\mathrm{Orth}(B)^{\sigma}$ of $\mathrm{Orth}(B)$.
\end{definition}

This definition is motivated by the following examples

\begin{enumerate}
\item Any composition operator $C\in\mathcal{L}^{r}(A,B)$ is a almost
contractive since It is contractive . To see this, notice first that $C$ is a
positive operator (see \cite[Theorem 4.3]{Triki}). Let $f\in U_{A}$ then
$f^{2}\leq f$, so $C(f)^{2}\leq C(f)$. Consequently, $C(f)\in U_{B}$. This is
implies that $C(U_{A})\subseteq U_{B}$ and $\omega_{C}\leq I_{B}$.

It follows that any weighted composition operator $T:A\longrightarrow B$ is
automatically almost contractive.

\item Let $T\in\mathcal{L}^{r}(A,B)$ be a separating operator. If the domain
$A$ has a unit element $I_{A}$, then $T$ is almost contractive and $\omega
_{T}=\left\vert T(I_{A})\right\vert $. Indeed, observe that if $f\in A$ with
$\left\vert f\right\vert \leq I_{A}$ then%
\[
\left\vert Tf\right\vert \leq\left\vert T\right\vert (\left\vert f\right\vert
)\leq\left\vert T\right\vert (I_{A})=\left\vert T(I_{A})\right\vert \text{.}%
\]

\end{enumerate}

In contrary to the unital case, a separating operator between semiprime
$f$-algebras may not be almost contractive as shown in the following example

\begin{example}
(see Example 6.6 in \cite{BBT}). Let $A$ be the $f$-algebra of the piecewise
polynomial functions on $\mathbb{R}_{\geq0}$ that are $0$ at $0$. Let $T$ be
the real-valued operator defined by $T(f)=f_{d}^{\prime}(0)$ for all $f\in A$.
It is easy to check that $T$ is a positive separating operator. However, $T$
is not almost contractive operator. Indeed, for all $n\in\mathbb{N}$, let
$f_{n}\in U_{A}$ defined by
\[
f_{n}(x)=%
\genfrac{\{}{.}{0pt}{}{nx\text{ if }0\leq x\leq\frac{1}{n}}{1\text{ if }%
x\geq\frac{1}{n}}%
\]
Then $T(f_{n})=n$ for all $n\in\mathbb{N}$.
\end{example}

In spite of that, we shall prove that any separating operator from $C_{0}(X)$
into $C_{0}(Y)$, for locally compact Hausdorff spaces $X$ and $Y$, is
automatically almost contractive

\begin{proposition}
\label{P1}Let $T\in\mathcal{L}^{r}(C_{0}(X),C_{0}(Y))$ be a separating
operator, then $T$ is almost contractive.
\end{proposition}

\begin{proof}
Suppose that the set $\left\vert T(U_{C_{0}(X)})\right\vert $ is not bounded
on $C_{b}(Y)=\mathrm{Orth}(C_{0}(Y))$, then there exist a sequence
$(f_{n})_{n\in\mathbb{N}}$ and a sequence $(y_{n})_{n\in\mathbb{N}}$ of $Y$
such that $0\leq f_{n}\leq1$ and $\left\vert T(f_{n})(y_{n})\right\vert \geq
n^{3}$. Define the uniform convergent sum%
\[
g(x)=%
{\displaystyle\sum\limits_{k=1}^{\infty}}
\frac{f_{k}(x)}{k^{2}}\text{ \ \ for all }x\in X.
\]
Obviously $g\in C_{0}(X)$. On the other hand, we have for all $n\in\mathbb{N}%
$,
\[
n\leq\frac{\left\vert T(f_{n})(y_{n})\right\vert }{n^{2}}\leq\left\vert
T(g)(y_{n})\right\vert
\]
a contradiction, since $T$ is norm continuous.
\end{proof}

\begin{proposition}
\label{P}Let $T\in\mathcal{L}^{r}(A,B)$ be a separating operator. $T$ is
almost contractive if and only if the positive and the negative parts $T^{+}$
and $T^{-}$ are almost contractives. Moreover, $\omega_{T^{+}}$ and
$\omega_{T^{-}}$ are components of $\omega_{T}$.
\end{proposition}

\begin{proof}
The equality $\left\vert T(f)\right\vert =T^{+}(f)+T^{-}(f)$ for all $f\in
U_{A}$ implies that $T^{+}$ and $T^{-}$ are almost contractives and
$\omega_{T}\leq\omega_{T^{+}}+\omega_{T^{-}}$. Furthermore, for all $f,g\in
U_{A}$ we have
\[
T^{+}(f)+T^{-}(g)\leq T^{+}(f\vee g)+T^{-}(f\vee g)=\left\vert T\right\vert
(f\vee g)\leq\omega_{T}%
\]
we derive easily that $\omega_{T^{+}}+\omega_{T^{-}}\leq\omega_{T}$. It
remains to prove that $\omega_{T^{+}}\wedge\omega_{T^{-}}=0$. To this end,
observe first that $T^{+}(f)\wedge T^{-}(g)=0$ for all $0\leq f,g\in U_{A}$.
This is follows from the inequalities
\[
0\leq T^{+}(f)\wedge T^{-}(g)\leq T^{+}(f+g)\wedge T^{-}(f+g)=0\text{.}%
\]
So, $\omega_{T^{+}}\wedge T^{-}(g)=0$ for all $g\in U_{A}$ which implies that
$\omega_{T^{+}}\wedge\omega_{T^{-}}=0$.
\end{proof}

The following lemma will be used in what follows, see \cite[Theorem 4]{Jaber}
for the proof.

\begin{proposition}
Let $T:A\longrightarrow B$ be a regular operator between semiprime
$f$-algebras $A$ and $B$. Then $T=0$ if and only if $T(f)=0$ for all $f\in
U_{A}$.
\end{proposition}

\section{Characterization of weighted composition opertaors}

Recall once again that $A$ and $B$ are semiprime $f$ -algebras. The symbole
$\mathfrak{C}(A,B)$ is used to indicate the set of all composition operators
in $\mathcal{L}^{r}(A,B)$. Recall that an operator $C\in\mathfrak{C}(A,B)$ if
\[
C(fg)=C(f)C(g)\text{ for all }f,g\in A\text{.}%
\]
Of course, any composition operators between semiprime $f$-algebras is a
positive and separating operator. The following characteriztion of composition
operator can be found in \cite[Theorem 8]{Jaber} and it will be used in the
proof of our main result.

\begin{theorem}
\label{Jaber}Let $C$ $\in\mathcal{L}^{r}(A,B)$. Then the following are equivalents:

\begin{enumerate}
\item[(i)] $C$ is a composition operator

\item[(ii)] $C$ is a positive separating and almost contractive operator with
$\omega_{C}$ is idemptent.
\end{enumerate}
\end{theorem}

Recall that an operator $T\in\mathcal{L}^{r}(A,B)$ is called a weighted
composition operator if there exist $\omega\in B$ and a composition operator
$C$ such that $T=\omega C$. Obviously, any weighted composition operator in
$\mathcal{L}^{r}(A,B)$ is both a separating and almost contractive operator.
Example 1.4 in \cite{AAB} shows that a separating and almost contractive
operator need not be a weighted composition operator. In spite of this, we
shall obtain a quite satisfactory condition on $\omega_{T}$ for a separating
and almost contractive operator $T$ to be a weighted composition operator.

Recall that a vector $f\in B$ is said to be von Neumann regular if there exist
$g\in B$ such that $fg^{2}=f$. To prove the central result of this section we
need the following lemma.

\begin{lemma}
\label{L}Let $\omega$ be a positive von Neumann element in a semiprime
$f$-algebra $B$ then any component of $\omega$ is von Neumann regular.
\end{lemma}

\begin{proof}
Let $0\leq\omega_{1}\leq\omega$ be a component of $\omega$. There exist $0\leq
v\in B$ such that $\omega^{2}v=\omega$. Let $\omega_{2}=\omega-\omega_{1}$ and
observe that $\omega_{1}\omega_{2}=0$ so, $\omega^{2}=\omega_{1}^{2}%
+\omega_{2}^{2}$. The equality $\omega^{2}v=\omega$ implies that $\omega
_{1}^{2}v-\omega_{1}=\omega_{2}-\omega_{2}^{2}v$. Hence
\[
(\omega_{1}^{2}v-\omega_{1})^{2}=(\omega_{1}^{2}v-\omega_{1})(\omega
_{2}-\omega_{2}^{2}v)=0\text{.}%
\]
Whence, $\omega_{1}^{2}v=\omega_{1}$ and $\omega_{1}$ is von Neumann regular.
\end{proof}

One can see an operator $T\in\mathcal{L}^{r}(A,B)$ as a regular operator from
the semiprime $f$-algebra $A$ to the Dedekind complet unital $f$-algebra
$\mathrm{Orth}(B)^{\sigma}$, so we will suppose that $B$ is a Dedekind
complete $f$-algebra with unit element $I_{B}$.

We have gathered thus all the ingredients we need for the proof of the central
result of this section.

\begin{theorem}
\label{Cent}Let $T\in\mathcal{L}^{r}(A,B)$ be an almost contractive operator
such that $\omega_{T}$ is a von Numan regular. Then $T$ is separating operator
if and only if there exists a composition operator $C$ $\in\mathfrak{C}(A,B)$
such that $T=\omega C$ with $\omega=\omega_{T^{+}}-\omega_{T^{-}}$.
\end{theorem}

\begin{proof}
We first assume that $T$ is positive almost contractive operator such that
$\omega_{T}$ is a von Numan regular. Since $\omega_{T}$ is positive then there
exist $0\leq v\in B$ such that $\omega_{T}^{2}v=\omega_{T}$. We claim that the
linear operator $C=vT$ is a composition operator. Clearly, $C$ is a positive
and a separating map. Moreover,
\[
\omega_{C}=\sup_{f\in U_{A}}C(f)=v\omega_{T}.
\]

Observe now that $\omega_{C}^{2}=v^{2}\omega_{T}^{2}=\omega_{C}$. This means
that $\omega_{C}$ is an idempotent element, which together with Theorem
\ref{Jaber} yields that $C$ is a composition operator. We claim that
$T=\omega_{T}C$. For this end, pick $f\in U_{A}$. We have
\[
(T-\omega_{T}C)f=Tf-\omega_{T}vTf=(I_{B}-\omega_{T}v)Tf
\]
So,
\[
\sup_{f\in U_{A}}\left\vert (T-\omega_{T}C)f\right\vert =\left\vert
I_{B}-\omega_{T}v\right\vert \omega_{T}=\left\vert \omega_{T}-\omega_{T}%
^{2}v\right\vert =0
\]
Thus $(T-\omega_{T}C)f=0$ for all $f\in U_{A}$. It follows that $T-\omega
_{T}C=0$.

Let's discuss the general case. Since $T$ is almost contractive and
$\omega_{T}$ is von Nuemann element, using proposition \ref{P} and lemma
\ref{L} we derive that $T^{+}$ and $T^{-}$ are almost contractive operators.
Moreover, $\omega_{T^{+}}$ and $\omega_{T^{-}}$ are von Nuemann regular
elements. By the positive case, we derive that there exist $v_{1},v_{2}\in B$
such that the mappings $C_{1}=v_{1}T^{+}$ and $C_{2}=v_{2}T^{-}$ are
composition operators and $T^{+}=\omega_{T^{+}}C_{1}$ and $T^{-}=\omega
_{T^{-}}C_{2}$. We claim that

\begin{description}
\item[(i).] $C=C_{1}+C_{2}$ is a composition operator

\item[(ii).] $T=(\omega_{T^{+}}-\omega_{T^{-}})C$.
\end{description}

To prove (i) observe that
\[
C_{1}(f)C_{2}(f)=v_{1}v_{2}T^{+}(f)T^{-}(f)=0\text{ for all }f\in A
\]
It follows that for all $f\in A$,
\[
C(f^{2})=C_{1}(f^{2})+C_{2}(f^{2})=C_{1}(f)^{2}+c_{2}(f)^{2}=(C_{1}%
(f)+C_{2}(f))^{2}=C(f)^{2}\text{.}%
\]
To prove (ii), pick $0\leq f\in A$. The equality $T^{+}(f)\wedge T^{-}(g)=0$
holds for all $g\in U_{A}$, then $\omega_{T^{-}}\wedge T^{+}(f)=0$ which
implies that $\omega_{T^{-}}C_{1}(f)=0$. The same argument shows that
$\omega_{T^{+}}C_{2}(f)=0$ for all $0\leq f\in A$. Hence,
\[
(\omega_{T^{+}}-\omega_{T^{-}})C=\omega_{T^{+}}C_{1}-\omega_{T^{-}}C_{2}%
=T^{+}-T^{-}=T\text{.}%
\]
This completes the proof.
\end{proof}

Keeping in mind that any separating operator $T:A\longrightarrow B$ with $A$
is unital $f$-algebra is automatically almost contractive and applying the
above Theorem, we obtain the Theorem 4.2 in \cite[Theorem 4.2]{AAB}.

\begin{corollary}
Let $A$ be an $f$-algebra with unit element $e$ and $B$ be a semiprime
$f$-algebra. Let $T\in\mathcal{L}^{r}(A,B)$ with $Te$ von Neumann regular.
Then T is separating if and only if there exists $C\in\mathfrak{C}(A,B)$ such
that $T=TeC.$
\end{corollary}

\begin{proof}
It is easy to see that since $Te$ is von Neumann regular then $\omega
_{T}=\left\vert Te\right\vert $ is a von Neumann regular element. According to
Theorem \ref{Cent} we derive that there exist $C\in\mathfrak{C}(A,B)$ such
that $T=(T^{+}e-T^{-}e)C=TeC$.
\end{proof}

Recall that if $B$ is semiprime $f$-algebra then \textrm{Orth}$(B)^{\sigma}$
is a Dedekind complete unital $f$-algebra. Moreover, the maximal ring of
quotient $\mathrm{Q(}$\textrm{Orth}$(B)^{\sigma})$ is now a Von Nuemann
regular $f$-algebra, that is, all elements are von Neumann regular . Moreover,
$B$ is an $f$-subalgebra of $\mathrm{Q(}$\textrm{Orth}$(B)^{\sigma})$. This
leads, via Theorem \ref{Cent}, to the following.

\begin{theorem}
\label{F}Let $A$ and $B$ be a semiprime $f$-algebra and $T\in\mathcal{L}%
^{r}(A,B)$ be a separating map. Then the following are equivalents:

\begin{enumerate}
\item[(i)] $T$ is almost contractive operator

\item[(ii)] There exist $\omega\in\mathrm{Orth(B)}^{\sigma}$ and $C$
$\in\mathfrak{C}(A,\mathrm{Q(Orth(B)}^{\sigma}))$ such that $T=\omega C$.
\end{enumerate}
\end{theorem}

Using proposition \ref{P1} and Theorem \ref{F} we get an algebraic version of
Jeang-Wong Theorem (see for instance \cite[Theorem 1]{Jeang-wong})

\begin{corollary}
Let $X$ and $Y$ be locally compact Hausdorff spaces and $T$ be a regular
separating map from $C_{0}(X)$ into $C_{0}(Y)$. Then there exists $\omega\in
C_{b}(Y)^{\sigma}$ and a composition operator $S:C_{0}(X)\longrightarrow
Q(C_{b}(Y)^{\sigma})$ such that $T(f)=\omega S(f)$ for all $f\in C_{0}(X)$.
\end{corollary}

At the end of this section we will give a complete answer to problem 6.1in
\cite{AAB}.

\begin{theorem}
(see \cite[Theorem 6.1]{AAB}) Let $B$ be a semiprime $f$-algebra and $p\in
B^{+}$. Then the following assertions are equivalents

\begin{description}
\item[(i)] $p$ is von Neumann regular.

\item[(ii)] For every $f$-algebra $A$ with identity $e$ and every separating
operator $T\in\mathcal{L}^{r}(A,B)$ with $p=Te$, there exists $C$
$\in\mathfrak{C}(A,B)$ such that $T=pC$ and $pB=CeB$.
\end{description}
\end{theorem}

\begin{proof}
The implication (i) implies (ii) is already proved in \cite{AAB}. We prove
only the converse. Let $A$ be an $f$-algebra with identity $e$ and
$T\in\mathcal{L}^{r}(A,B)$ be a separating operator with $p=Te$. By hypothesis
there exists $C$ $\in\mathfrak{C}(A,B)$ such that $T=pC$ and $pB=CeB$. We have
$p=Te=pCe$. On the other hand, there exists a vector $v\in B$ such that
$Ce=(Ce).Ce=pv$. Combining the two equalities we get that $p=p^{2}v$.
\end{proof}

\section{Vector spaces of almost contractive separating operators}

We call a vector space of almost contractive separating regular operators from
$A$ into $B$ any vector subspace $W$ of $\mathcal{L}^{r}(A,B)$ the operators
in which are almost contractive and separating. Our main purpose is to give a
complete description of such vector spaces. 

It has been proved in \cite[Theorem 5.1]{AAB} that if $A$ is supposed to have
a unit element then any vector space of separating regular operators from $A$
into $B$ is contained in a one-dimensional $B$-submodule of the $B$-module
$\mathcal{L}^{r}(A,$\textrm{Q(}$B))$ generated by some composition operator
$C$ from $A$ into $\mathrm{Q}(B)$.

As we shall see, this result can be extended for all vector spaces of almost
contractive separating regular operators between semiprime $f$-algebras.

\begin{theorem}
A subset $\mathfrak{M}$ of $\mathcal{L}^{r}(A,B)$ is a vector space of almost
contractive separating regular operators from $A$ into $B$ if and only if
there exist $C$ $\in\mathfrak{C}(A,\mathrm{Q(Orth(B)}^{\sigma}))$ such that
for all $T\in\mathfrak{M}$ there exist $\omega\in\mathrm{Orth(B)}^{\sigma}$
such that $T=\omega C$. 
\end{theorem}

\begin{proof}
Let $\mathfrak{M}$ be vector space of $\mathcal{L}^{r}(A,B)$ such that all
operators in $\mathfrak{M}$ are almost contractive separating operators. We
shall see $\mathfrak{M}$ as a vector subspace of $\mathcal{L}^{r}%
(A,\mathrm{Q(Orth(B)}^{\sigma}))$. It is easy to see that if $T,S$
$\in\mathfrak{M}$ then $TfSg=0$ for all $f$,$g\in A$ such that $fg=0$.

Examining the first part of the proof of Theorem 5.1 in \cite{AAB} we can
derive that $\mathfrak{M}$ is contained in a maximal vector lattice denoted
egain by $\mathfrak{M}$ where the lattice operations are given pointwise.
Moreover, $\mathfrak{M}$ has a positive weak order unit $E$ $\in
\mathcal{L}^{r}(A,\mathrm{Q(Orth(B)}^{\sigma}))$. From Theorem \ref{F}, there
exists $C$ $\in\mathfrak{C}(A,\mathrm{Q(Orth(B)}^{\sigma}))$ such that
$E=\omega_{E}C$.

We claim that if $T\in\mathfrak{M}$ then $T=\omega C.$ For this end, let $T$
be a positive operator in $\mathfrak{M}$. Observe that the bilinear mapping
$b:A\times A\longmapsto\mathrm{Q(Orth(B)}^{\sigma})$ defined by
$b(f,g)=E(f)T(g)$ is a positive orthosymmetric bilinear map. So $b$ is now a
symmetric bilinear map (see \cite{BV}). Hence, the equality
\[
E(f)T(g)=E(g)T(f)\text{ holds for all }f,g\in A\text{.}%
\]
We get,
\[
\omega_{E}T(g)=\sup_{f\in U_{A}}E(f)T(g)=\sup_{f\in U_{A}}E(g)T(f)=\omega
_{T}E(g)\text{ for all }g\in A\text{.}%
\]
Thus, $\omega_{E}(T(f)-\omega_{T}C(f))=0$ for all $f\in A$. Then we derive
that
\[
\omega_{E}\wedge\left\vert T-\omega_{T}C\right\vert (f)=0\text{ for all }f\in
A\text{.}%
\]
But, this implies that $E\wedge\left\vert T-\omega_{T}C\right\vert =0$.
Consequently, $T=\omega_{T}C$ because $E$ is a weak order unit in
$\mathfrak{M}$.

Let $T\in\mathfrak{M}$ since $T=T^{+}-T^{-}$ and $T^{+}$ and $T^{-}$ belongs
to $\mathfrak{M}$, we have $T=(\omega_{T^{+}}-\omega_{T^{-}})C$. Putting
\[
W=\left\{  \omega:\text{ }T=\omega C\text{ for some }T\in\mathfrak{M}\right\}
\]
we conclude that $W$ is a vector subspace of $B$ and that
\end{proof}

\end{document}